\documentclass{amsart}
\usepackage{amsmath,graphicx,amsfonts,subcaption,comment,amssymb,amsthm,xcolor,hyperref} 
\usepackage{booktabs} 

\usepackage[
  backend=biber,
  style=numeric,
  sorting=nyt,
  maxnames=99,
  giveninits=true,
  eprint=true
]{biblatex}
\renewbibmacro{in:}{}



\definecolor{prussianblue}{rgb}{0.1,0.2,0.6}

\newtheorem{corollary}{Corollary}

\def\ti{\mathrm{i}}
\def\Li{\operatorname{Li}}
\def\Re{\operatorname{Re}}
\def\Im{\operatorname{Im}}
\def\euler{\genfrac{\langle}{\rangle}{0pt}{}}

\title{A new product formula for $(z;q)_\infty$,\\ with applications to asymptotics}

\begin{document}

\author{Arash Arabi Ardehali}
\address{Department of Physics\\ Sharif University of Technology\\
P.O. Box 11155-9161, Tehran, Iran}
\email{a.a.ardehali@gmail.com}
\author{Hjalmar Rosengren}
\address
{Department of Mathematical Sciences
\\ Chalmers University of Technology and University of Gothenburg\\SE-412~96 G\"oteborg, Sweden}
\email{hjalmar@chalmers.se}

\begin{abstract}
    \noindent We express the $q$-Pochhammer symbol $(z;q)_\infty$ as an infinite product of gamma functions, analogously to how Narukawa expressed the elliptic gamma function as an infinite product of hyperbolic gamma functions. This identity is used to obtain  asymptotic expansions when  $q$ tends to $1$.
\end{abstract}

\maketitle

\tableofcontents

\section{Introduction}
The infinite $q$-Pochhammer symbol
\begin{equation}\label{q-pochhammer}(z;q)_\infty=\prod_{j=0}^\infty(1-zq^j)\end{equation}
plays a fundamental role for $q$-deformations of special functions.
For instance, the limits
\begin{align*}\lim_{q\rightarrow 1} (z(q-1);q)_\infty&=e^z,\\
\lim_{q\rightarrow 1}\frac{(q;q)_\infty(1-q)^{1-z}}{(q^z;q)_\infty}&=\Gamma(z),\\
\lim_{q\rightarrow 1}(q-1)\log (z;q)_\infty&=\Li_2(z):=\sum_{n=1}^\infty\frac{z^n}{n^2},
\end{align*}
show that it can be viewed as a $q$-analogue of the exponential function, the gamma function or the dilogarithm, depending on how the variable $z$ is rescaled.

The main result of the present work is the identity
\begin{multline}\label{main-identity-convergent}
    (e^{-y};e^{-\beta})_\infty
    = \exp\left(\frac{\beta(1-{{\coth(y/2)}})}{24}-\frac{\Li_2(e^{-y})}\beta\right) (1-e^{-y})^{1/2}\\
   \times \prod_{n\in\mathbb
    {Z}}\left(\frac{\sqrt{2\pi}}{\Gamma\left(\frac{y+2\pi \ti n}{\beta}\right)} \left(\frac{y+2\pi \ti n}{\beta}\right)^{\frac{y+2\pi \ti n}{\beta}-\frac 12} \exp\left({\frac{\beta}{12(y+2\pi\ti  n)}-\frac{y+2\pi \ti n}{\beta}}\right)\right),
\end{multline}
which holds for $\Re \beta>0$ and $y\notin 2\pi\ti\mathbb Z$. 
As we discuss below (see \eqref{dedekind}), the case $y=\beta$  is essentially
 the modular transformation 
 for Dedekind's eta function. The identity
 \eqref{main-identity-convergent} generalizes this  transformation to the non-modular function $(e^{-y};e^{-\beta})_\infty$.

 The dilogarithm and non-integer powers on the right-hand side of \eqref{main-identity-convergent} have branch  
 points at $y\in 2\pi\ti\mathbb Z$. 
 The identity holds for the standard branches,
with $\Li_2(e^{-y})$ and $(1-e^{-y})^{1/2}$  positive for $y>0$ and $((y+2\pi\ti n)/\beta)^{(y+2\pi\ti n)/\beta-1/2}$ positive for $(y+2\pi\ti n)/\beta>0$. 
Since the left-hand side is entire in $y$,  the total monodromy around each branch point vanishes, which can be verified by direct computation. 

The identity \eqref{main-identity-convergent} is closely related to 
Stirling's approximation 
\begin{equation}\label{stirling}\log\Gamma(x)=\left(x-\frac 12\right)\log(x)-x+\frac 12\log(2\pi)+\frac 1{12 x}+f(x),\end{equation}
where  $f(x)=\mathcal O(x^{-3})$ as $x\rightarrow\infty$ in any sector avoiding the negative axis. Indeed, the product on the right-hand side of 
\eqref{main-identity-convergent} is 
\begin{equation}\label{exp-sum-f}\exp\left(-\sum_{n\in\mathbb Z} f\left(\frac{y+2\pi\ti n}\beta\right)\right).\end{equation}
In particular, it is absolutely convergent.
We will give two proofs of \eqref{main-identity-convergent}, using different expressions for $f$.
Binet's identity \cite[Eq.\ (1.7.23)]{bateman} 
\begin{equation}\label{binet}f(x)=-\frac 1{12x}+\int_0^\infty \left(\frac 1{1-e^{-t}}-\frac 1t-\frac 12\right)\frac{e^{-tx}}{t}\,dt, \end{equation}
can be used to compute the Fourier transform of $f$. 
A straight-forward application of Poisson's summation formula then leads to \eqref{main-identity-convergent}, see \S \ref{sec:poisson-proof}.
In \S \ref{sec:proof} we give an alternative proof based on Artin's identity  \cite[Eq.\ (3.9)]{Artin}
\begin{equation}\label{artin-sum}f(x)=-\frac 1{12x}+ \sum_{n=0}^\infty\left(\left(x+n+\frac 12\right)\log\frac {x+n+1}{x+n}-1\right).
\end{equation}

The formula \eqref{main-identity-convergent} expresses the left-hand side as an infinite product of reciprocal gamma functions (dressed with additional factors), analogously to how Narukawa \cite{narukawa2004modular} expressed the elliptic gamma function as an infinite product of dressed hyperbolic gamma functions, see \eqref{narukawa}. In fact, as we will show in Appendix \ref{app:narukawa}, the identity~\eqref{main-identity-convergent} can
at least formally
be obtained as a limit of Narukawa's identity. However, this involves a non-trivial regularization of the terms, and working out the details turns out to be more complicated than proving \eqref{main-identity-convergent} from scratch. 

The two mentioned product formulas have
the following quantum field theory (QFT) interpretation. Narukawa's formula  expresses the BPS partition function of a 4d $\mathcal{N}=1$ multiplet on $S^3\times S^1$ as an infinite product of the BPS partition functions on $S^3$ of the 3d $\mathcal{N}=2$ chiral multiplets arising from Fourier (or, in physics terminology, Kaluza--Klein) expansion of the 4d chiral multiplet around $S^1$, see e.g.~\cite{Aharony:2013dha}. Our formula analogously expresses the BPS partition function of a 3d $\mathcal{N}=2$ chiral multiplet on $D^2\times S^1$ (\cite{Beem:2012mb,Yoshida:2014ssa}) as an infinite product of the BPS partition functions on $D^2$ of the 2d $\mathcal{N}=(2,2)$ chiral multiplets arising from Fourier (or Kaluza--Klein) expansion of the 3d chiral multiplet around $S^1.$ A formula following from \eqref{main-identity-convergent} by taking ratios of the two sides at different $y$'s, expressing the $S^2\times S^1$ partition function as an infinite product of $S^2$ partition functions, was given without proof in Eq.~(A.21) of \cite{Aharony:2017adm}.
We explain in Appendix~\ref{app:narukawa} how a geometric degeneration of $S^3\times S^1$ to $D^2\times S^1$ underlies the reduction of Narukawa's formula to our identity.


Our main motivation for seeking identities of the form \eqref{main-identity-convergent} has been the need to study the $q\to1$ asymptotics of $(z;q)_\infty$, where $z$ is allowed to vary with $q$. In particular, the equivalent identity \eqref{main-identity} is used in \cite{ArabiArdehali:2025bub} to analyze the $q \to 1$ limit of basic hypergeometric
integrals (see e.g.\ 
\cite{Gahramanov:2016wxi,Krattenthaler:2011da,rosengren2018rahman} for examples relevant to the physical context described above).
This parallels the role played by Narukawa’s identity in
\cite{ArabiArdehali:2018ngl,Ardehali:2021irq,Rains:2006dfy}
in the study of the hyperbolic limit of elliptic hypergeometric integrals.
In \S \ref{sec:asymptotics}, we will use \eqref{main-identity-convergent} to obtain a uniform
asymptotic expansion of $(e^{-x};e^{-\beta})_\infty$
as $\beta\rightarrow 0$. This is then specialized to give 
the complete asymptotics of
$\log (e^{-x\beta^c};e^{-\beta})_\infty$, where $c\geq 0$. The cases $c=0$ and $c=1$ have been previously
studied in the literature, as we review below. By contrast, our results for the regimes
$0 < c < 1$ and $c > 1$ appear to be new.
Among these, the value $c = \tfrac{1}{2}$ is of particular interest, since a central
conjecture of \cite{ArabiArdehali:2025bub} asserts that the dominant scaling regimes in the
$q \to 1$ asymptotics of basic hypergeometric integrals correspond precisely to
$c \in \{0, \tfrac{1}{2}, 1\}$. We hope that the results obtained here will contribute to a
resolution of this conjecture as well. In \S \ref{sec:numerics} we
study the truncation error in our asymptotic expansions. This part of the paper is not mathematically rigorous, but based on heuristic arguments and computer experiments.



{\bf Acknowledgements:} We would like to thank the anonymous referees for several useful comments. We also thank Davide Cassani and Cyril Closset for explaining to us the details of the degeneration from $S^3\times S^1$ to $D^2\times S^1$ discussed in Appendix \ref{app:narukawa}.

\section{First proof of the main identity}
\label{sec:poisson-proof}

Our first proof of \eqref{main-identity-convergent} is based on
Poisson's summation formula
\begin{equation}\label{poisson}
\sum_{n\in\mathbb Z}\, \int_{\mathbb R} \phi(t)e^{-2\pi \ti n t} dt =  \sum_{n\in\mathbb Z} \phi(n).
\end{equation}
A sufficient condition for \eqref{poisson} to hold is that $\phi$ is continuous and the terms on both sides decay as $\mathcal O(|n|^{-1-\varepsilon})$ for some $\varepsilon>0$ \cite[\S 8.32]{folland}. If $\phi$ is defined on $\mathbb R_{\geq 0}$, applying \eqref{poisson} to the even extension of $\phi$ gives 
\begin{equation}\label{poisson-even}
\sum_{n\in\mathbb Z}\, \int_{0}^\infty \phi(t)e^{-2\pi \ti n t} dt =\frac{\phi(0)}2 + \sum_{n=1}^\infty \phi(n).
\end{equation}
We will use this with
$$\phi(t)=\left(\frac 1{\beta t}+\frac 12+\frac{\beta t}{12}-\frac 1{1-e^{-\beta t}}\right)\frac{e^{- y t }}{t},$$
where we  assume $\beta>0$ and $\Re y>0$.
Note that  $\phi$ is continuous at $0$ with $\phi(0)=0$ and decays exponentially at $+\infty$.

By Binet's identity \eqref{binet}, the integral in
\eqref{poisson-even} equals $-f((y+2\pi\ti n)/\beta)$. In particular, the terms on the left decay as $\mathcal O(n^{-3})$, so \eqref{poisson-even} is valid.
On the right-hand side we use
\[
\sum_{n=1}^\infty \frac{e^{- n y}}{n(1-e^{-\beta n})} = \sum_{n=1}^\infty \sum_{k=0}^\infty \frac{e^{- n(y+\beta k)}}{n} = -\sum_{k=0}^\infty \log (1 - e^{-y-\beta k}).
\]
The other sums appearing are elementary, and after rearranging the terms we obtain
\begin{multline}\label{main-identity-log}\sum_{k=0}^\infty \log (1 - e^{-y-\beta k})\\
=-\frac{\Li_2(e^{- y})}\beta+\frac{\log(1-e^{- y})}{2}+\frac{\beta(1-\coth( y/2))}{24}-\sum_{n\in\mathbb Z}f\left(\frac{y+2\pi\ti n}\beta\right).\end{multline}
Exponentiating this identity gives \eqref{main-identity-convergent}, assuming $\beta >0$ and $\Re y>0$.

 It is known
that $|f(x)|\leq C|x|^{-3}$, uniformly in any sector $|\arg(x)|\leq \phi<\pi$. It follows that the final sum in \eqref{main-identity-log} converges locally uniformly for 
$\Re(\beta)>0$ and $y\notin 2\pi \ti\mathbb Z$. Hence, the limit function is analytic in this domain, apart from possible branch points at $y\in 2\pi\ti \mathbb Z$. It then follows by analytic continuation that 
 \eqref{main-identity-convergent} holds in the larger parameter domain.

\section{Alternative proof of the main identity}\label{sec:proof}

We will now give an alternative proof of \eqref{main-identity-convergent},
which is  longer but more elementary as it does not involve Fourier analysis or analytic continuation. We start from Artin's identity \eqref{artin-sum} in the form 
$$f(x)=\sum_{m=0}^\infty\phi(x+m), $$
where
$$\phi(x)=\left(x+\frac 12\right)\log\left(\frac{x+1}{x}\right)-1+\frac 1{12(x+1)}-\frac 1{12x}. $$
This holds for $x\notin\mathbb R_{\leq 0}$.
Since $\phi(x)=\mathcal O(x^{-4})$, the double sum
$$S=\sum_{n=-\infty}^\infty \sum_{m=0}^\infty\phi\left(\frac{y+2\pi\ti n}\beta+m\right)$$
  is absolutely convergent and we can interchange the  summations.
That is, 
$$S=\sum_{m=0}^\infty T(y+m\beta), $$
where
$$T(y)=\sum_{n\in\mathbb Z} \phi\left(\frac{y+2\pi\ti n}\beta\right).$$

To compute $T$, we  split $\phi=\phi_1+\phi_2$, where 
$$\phi_1(x)=\left(x+\frac 12\right)\log\left(\frac{x+1}{x}\right)-1,\qquad \phi_2(x)=\frac 1{12(x+1)}-\frac 1{12x}. $$
We apply the integral representation
$$\phi_1(x)=\frac 12\int_0^1u\left(\frac 1{2x+1-u}-\frac 1{2x+1+u}\right)\,du. $$
Interchanging sum and integal, which is allowed by Fubini's theorem, gives
\begin{align*}T_1(y)&=\sum_{n\in\mathbb Z} \phi_1\left(\frac{y+2\pi\ti n}\beta\right)\\
&=\frac 1 2\int_0^1 u\sum_{n\in\mathbb Z} \left(\frac 1{\frac{2(y+2\pi\ti n)}{\beta}+1-u}-\frac 1{\frac{2(y+2\pi\ti n)}{\beta}+1+u}\right)\,du\\
&=\frac\beta 8\int_0^1u\left(\coth\left(\frac{2y+\beta(1-u)}{4}\right)-\coth\left(\frac{2y+\beta(1+u)}{4}\right)\right)\,du,
\end{align*}
where we used \eqref{coth-parfrac} in the last step.
The integrand has the primitive function
\begin{multline*}\frac 8{\beta^2}\left(\Li_2\left(e^{-y-\frac{\beta(1+u)}2}\right)-\Li_2\left(e^{-y-\frac{\beta(1-u)}2}\right)\right)  \\
-\frac {4u}\beta\left(\log\left(1-e^{-y-\frac{\beta(1+u)}2}\right)+\log\left(1-e^{-y-\frac{\beta(1-u)}2}\right)\right),\end{multline*}
which leads to
$$T_1(y)=\frac 1\beta\left(\Li_2(e^{-y-\beta})-\Li_2(e^{-y})\right)-\frac 12\left(\log(1-e^{-y-\beta})+\log(1-e^{-y})\right). $$
Another application of \eqref{coth-parfrac} gives
$$T_2(y)=\sum_{n\in\mathbb Z} \phi_2\left(\frac{y+2\pi\ti n}\beta\right)=\frac\beta{24}\left(\coth\left(\frac{y+\beta}2\right)-\coth\left(\frac y2\right)\right). $$

We have now proved that
$$S=\sum_{m=0}^\infty (T_1+T_2)(y+m\beta)= \sum_{m=0}^\infty (a_{m+1}+a_m+b_{m+1}-b_m),$$
where 
$$a_m=-\frac 12\log(1-e^{-y-m\beta}), $$
$$b_m=\frac 1\beta\Li_2(e^{-y-m\beta})+\frac\beta{24}\coth\left(\frac{y+m\beta}{2}\right).$$
We rewrite the partial sums as
$$\sum_{m=0}^N (a_{m+1}+a_m+b_{m+1}-b_m)=a_{N+1}+b_{N+1}-a_0-b_0+2\sum_{m=0}^Na_m$$
and then let $N\rightarrow\infty$.
This gives
$$S=\frac{\beta(1-\coth(y/2))}{24}+\frac 12\log(1-e^{-y})-\frac 1\beta\Li_2(e^{-y})-\sum_{m=0}^\infty\log(1-e^{-y-m\beta}),  $$
which is equivalent to \eqref{main-identity-convergent}.

\section{Applications}
\label{sec:asymptotics}

To shorten some identities below, we will use the normalized gamma function
$$\Gamma^\ast(x)=\frac{\Gamma(x)}{\sqrt{2\pi}x^{x-1/2}e^{-x}}.$$
In this notation, Stirling's asymptotic formula takes the form
\begin{equation}\label{stirling-general}\log\Gamma^\ast(x)=\sum_{k=1}^N\frac{B_{2k}}{2k(2k-1)x^{2k-1}}+f_N(x),
\end{equation}
where $B_j$ are Bernoulli numbers and
$f_N(x)=\mathcal O(x^{-2N-1})$ in any sector avoiding the negative axis. It may be worth pointing out that 
fast algorithms with precise error bounds are available for computing this function \cite{boyd}.

One can improve the convergence in  \eqref{exp-sum-f}  by replacing $f=f_1$ with a higher order remainder 
term $f_N$. As a compensation, one  needs to insert additional factors
$$\exp\left(\frac{B_{2k}}{2k(2k-1)}\sum_{n\in\mathbb Z}\left( \frac \beta{y+2\pi\ti n}\right)^{2k-1}\right)
=\exp\left( \frac{B_{2k}}{(2k)!}\,\Li_{2-2k}(e^{-y})\beta^{2k-1}\right),$$
where the identity holds for $k\geq 2$, see \eqref{polylog-parfrac}. 
It follows that 
\begin{multline}\label{fast-convergence-cotangent}
    (e^{-y};e^{-\beta})_\infty
    = \exp\left(\frac{\beta(1-\coth(y/2))}{24}+\frac 12\log(1-e^{-y})-\frac{\Li_2(e^{-y})}\beta\right.\\
    \left.-\sum_{k=2}^N\frac{B_{2k}\Li_{2-2k}(e^{-y})\beta^{2k-1}}{(2k)!}
    -\sum_{n\in\mathbb Z}f_N\left(\frac{y+2\pi\ti n}\beta\right)
    \right),
\end{multline}
Using that $\coth(y/2)=1+2\Li_0(e^{-y})$ and $\log(1-e^{-y})=-\Li_1(e^{-y})$,
this can be written  compactly as follows. The special case $N=1$ 
is  \eqref{main-identity-convergent}, whereas  $N\geq 2$ gives equivalent identities with improved convergence. 
Note that $B_{k}=0$ if $k\geq 3$ is odd.

\begin{corollary}
For $N\in\mathbb Z_{>0}$, $\Re \beta >0$ and $y\notin 2\pi \ti\mathbb Z$, 
\begin{equation}\label{fast-convergence}
    (e^{-y};e^{-\beta})_\infty
    = \exp\left(\sum_{k=0}^{2N}\frac{B_{k}\Li_{2-k}(e^{-y})(-\beta)^{k-1}}{k!}-\sum_{n\in\mathbb Z}f_N\left(\frac{y+2\pi\ti n}\beta\right)\right),
\end{equation}
where $f_N$ is defined by \eqref{stirling-general}. 
\end{corollary}

We can also  replace $f_1$ by $f_0$, which corresponds to deleting from \eqref{main-identity-convergent} the factor
$$\exp\left(\operatorname{PV}\sum_{n\in\mathbb Z}\frac{\beta}{12(y+2\pi\ti n)}
\right)=\exp\left(\frac{\beta\coth(y/2)}{24}\right),
$$
see \eqref{coth-parfrac}. This gives
\begin{equation}\label{main-identity}
    (e^{-y};e^{-\beta})_\infty
    = (1-e^{-y})^{1/2}\exp\left(\frac{\beta}{24}-\frac{\mathrm{Li}_2(e^{-y})}{\beta}\right)\operatorname{PV}\ \prod_{n\in\mathbb
    {Z}}\frac 1 {\Gamma^\ast(\frac{y+2\pi \ti n}{\beta})},
\end{equation}
which was announced and used in \cite{ArabiArdehali:2025bub}.

An  interesting consequence of \eqref{main-identity} follows by 
taking the quotient of the original identity and the one obtained after replacing $y$ with $y+\beta$. Both sides simplify and we obtain 
\begin{equation}\label{interesting-consequence}
 \prod_{n\in\mathbb Z}e^{-1}\left(1+\frac{\beta}{y+2\pi \ti n}\right)^{\frac{y+2\pi \ti n}\beta+\frac 12}=\frac{\exp\left(\frac{\mathrm{Li}_2(e^{-y-\beta})-\mathrm{Li}_2(e^{-y})}\beta\right)}{(1-e^{-y})^{1/2}(1-e^{-y-\beta})^{1/2}}.
\end{equation}
Since the factors are $1+\mathcal O(n^{-2})$, there is no need to take the principal value.  This can be compared with the product version of Artin's identity \eqref{artin-sum},
$$\prod_{n=0}^\infty e^{-1}\left(1+\frac 1{x+n}\right)^{x+n+\frac 12}=\Gamma^\ast(x).$$

The case $y=\beta$ of
\eqref{main-identity} is also interesting. 
Using
$$\frac 1{\Gamma\left(1+\frac{2\pi\ti n}\beta\right)\Gamma\left(1-\frac{2\pi\ti n}\beta\right)}=\frac{\beta}{4\pi^2 n}(e^{2\pi^2n/\beta}-e^{-2\pi^2n/\beta}),\qquad n\neq 0,$$
and
$$\prod_{n\neq 0}e^{-1}\left(1+\frac{\beta}{2\pi \ti n}\right)^{\frac{2\pi \ti n}\beta+\frac 12}=\frac{\exp\left(1-\frac{\pi^2}{6\beta}+\frac{\mathrm{Li}_2(e^{-\beta})}\beta\right)}{\beta^{1/2}(1-e^{-\beta})^{\frac 12}},$$
which is the limit case $y\rightarrow 0$
of \eqref{interesting-consequence}, gives after a short computation
\begin{equation}\label{dedekind}(e^{-\beta};e^{-\beta})_\infty=\sqrt{\frac{2\pi}\beta}\,e^{\frac{\beta}{24}-\frac{\pi^2}{6\beta}}(e^{-4\pi^2/\beta};e^{-4\pi^2/\beta})_\infty.\end{equation}
This is the modular transformation for Dedekind's eta function.
In a similar way, one obtains  modular transformations for Jacobi theta functions (cf. \eqref{theta-modular})
by multiplying two instances of \eqref{main-identity}.
Another explicit formula for $(x;q)_\infty$, which also implies the modular transformations of eta and theta functions, has been obtained by Zhang \cite{zhang}.

We now turn to applications to asymptotics. 
We will formulate the results in terms of
$\log(e^{-y};e^{-\beta})_\infty$, which requires
a careful choice of branches. Note first that
$(e^{-y};e^{-\beta})_\infty$ vanishes for
$y\in 2\pi\ti\mathbb Z+\beta \mathbb Z_{\leq 0}$. Hence, it has a well-defined logarithm outside the cuts
$ 2\pi\ti \mathbb Z+\beta \mathbb R_{\leq 0}$.
We pick the branch that vanishes as $y\rightarrow\infty$,
 given  by
the left-hand side of \eqref{main-identity-log}
for $\Re y>0$. 
It follows from that identity that we
can take the logarithm of
\eqref{fast-convergence} in the sense that
\begin{equation}\label{log-product}\log
(e^{-y};e^{-\beta})_\infty
    = \sum_{k=0}^{2N}\frac{B_{k}\Li_{2-k}(e^{-y})(-\beta)^{k-1}}{k!}-\sum_{n\in\mathbb Z}f_N\left(\frac{y+2\pi\ti n}\beta\right),\end{equation}
    where we choose branches so that each term on the right vanishes as $y\rightarrow\infty$.
    If we want to use \eqref{log-product}
    for $\Re y<0$, we need to continue $\Li_1(e^{-y})$ and $\Li_2(e^{-y})$ analytically to $ y\notin 2\pi\ti \mathbb Z+\beta \mathbb R_{\leq 0}$. This is different from the standard branches,
which are analytic for $y\notin 2\pi\ti \mathbb Z+ \mathbb R_{\leq 0}$.  

We will use \eqref{log-product} to study the asymptotics of 
$\log (e^{-y};e^{-\beta})_\infty$ when $\beta\rightarrow 0$ and $y$ may depend on $\beta$. 
Note that if
  $\beta$ approaches zero on a path tangential to the imaginary axis, then $(y+2\pi\ti n)/\beta$
 has argument close to $\pi$ for large $|n|$. We are then leaving the region where \eqref{stirling-general} is valid. To avoid such complications, we will assume that
 $\arg \beta\in(-\pi/2,\pi/2)$ is fixed. 
By periodicity in $y$, we can also assume that $|\Im y|\leq \pi$. The following result gives a uniform asymptotic expansion in this situation.

\begin{corollary}\label{corollary-uniform}
Fix $\theta$ 
and  $Y$ with $|\theta|<\pi/2$ and $0<Y<\pi|\cot\theta|$
(if $\theta=0$ we only require $Y>0$). 
Then, in the region
\begin{equation}\label{beta-y-region}|\Im y |\leq \pi,\quad \Re y \geq -Y,\quad y\notin\beta\mathbb R_{\leq 0},\quad \arg\beta=\theta,\end{equation}
one has the following uniform asymptotic expansion as $\beta\rightarrow 0$:
\begin{multline}\label{uniform-asymptotics}
    \log(e^{-y};e^{-\beta})_\infty
    \sim-\frac{\Li_2(e^{-{y}})}{\beta}+\frac{\log(1-e^{-{y}})}{2} -\log\Gamma^\ast\left(\frac{{y}}{\beta}\right)\\
   -\sum_{k=1}^{\infty}\frac{B_{2k}\beta^{2k-1}}{(2k)!}\left(\Li_{2-2k}\left(e^{-y}\right)-\frac{(2k-2)!}{{y}^{2k-1}}\right).
\end{multline}
Here, $\Li_2(e^{-{y}})$ and $\log(1-e^{-{y}})$
denote branches that are real-valued for $y>0$
and analytic for $y\notin\beta \mathbb R_{\leq 0}$, whereas $\log\Gamma^\ast(y/\beta)$
denotes the branch that is 
real-valued for $y/\beta>0$ and analytic for 
 $y\notin\beta \mathbb R_{\leq 0}$.
\end{corollary}

To clarify the statement, 
if we truncate the right-hand side of \eqref{uniform-asymptotics} at $k=N$, then the difference between the two sides can be estimated
by $D|\beta|^{2N+1}$, where $D=D(N,Y,\theta)$ is independent of $\beta$ and $y$ subject to \eqref{beta-y-region}.
The condition \eqref{beta-y-region} can be weakened but has been chosen for simplicity.

\begin{proof}
 As discussed above, 
\eqref{log-product} is valid as long as $y$ avoids 
$\mathbf C_n:=2\pi\ti n+\beta \mathbb R_{\leq 0}$ for all $n\in\mathbb Z$. The cut $\mathbf C_0$ intersects the strip 
$|\Im y |\leq\pi$, which leads to the condition $y\notin\beta\mathbb R_{\leq 0}$. If $\theta=\arg\beta>0$, then  $\mathbf C_n$ also intersect the strip if $n>0$. The cut $\mathbf C_1$ enters the strip at the point
$-\pi\cot\theta+\ti\pi$ and the other ones do so further to the left. Similarly, if $\theta<0$, $\mathbf C_{-1}$ enters at $\pi\cot\theta-\pi\ti$ and the  cuts with $n<-1$  further to the left. Hence, the condition $\Re y\geq -Y>-\pi|\cot\theta|$ is sufficient to avoid all cuts with $n\neq 0$.  

We write the right-hand side of \eqref{log-product} as
\begin{equation}\label{uniform-decomposition}
\sum_{k=0}^{2N}\frac{B_{k}\Li_{2-k}(e^{-y})(-\beta)^{k-1}}{k!}
-
f_N\left(\frac y\beta\right)
-\sum_{n\neq 0}f_N\left(\frac{y+2\pi\ti n}\beta\right),\end{equation}
where the final sum is viewed as a remainder term. 
It is known that $|f_N(x)|\leq C|x|^{-2N-1}$, uniformly in any sector $|\arg x|\leq\phi<\pi$. To apply this estimate, $x$ must avoid a sector along each $\mathbf C_n$ with $n\neq 0$.
Since we have strict inequality in $Y<\pi|\cot\theta|$, it follows from the discussion above that this can be achieved with the same angle chosen for each sector. We can then estimate
\begin{equation}\left|\sum_{n\neq 0}f_N\left(\frac{y+2\pi\ti n}{\beta}\right)\right|\leq
 C|\beta|^{2N+1} \sum_{n\neq 0} \frac 1{|y+2\pi \ti n|^{2N+1}}.
 \end{equation}
Since 
$|\Im y|\leq \pi$ implies $|y+2\pi\ti n|\geq(2|n|-1)\pi$, the latter sum is bounded independently of $y$.
It follows that
\begin{equation}\label{log-product-estimate}
\log(e^{-y};e^{-\beta})_\infty
=\sum_{k=0}^{2N}\frac{B_{k}\Li_{2-k}(e^{-y})(-\beta)^{k-1}}{k!}-
f_N\left(\frac y\beta\right)+R_N,
\end{equation}
with $|R_N|\leq D|\beta|^{2N+1}$, where $D$ is independent of $\beta$ and $y$.
 Inserting the explicit expression for $f_N$ from \eqref{stirling-general} gives \eqref{uniform-asymptotics}. 
\end{proof}

In view of its uniform nature,
Corollary \ref{corollary-uniform} 
can be used to obtain the asymptotics of $\log (e^{-y};e^{-\beta})_\infty$ in various scaling regimes $y=y(\beta)$. 
As an illustration, we  work out the complete asymptotics 
 when $y\sim \beta^c$ with 
$c\geq 0$. As we discuss below, the cases $c=0$ and $c=1$ are known
from the literature, but even then our method is new.  

\begin{corollary}\label{theorem-c>=0}
Let $c\geq 0$  and fix a value of $\arg\beta\in(-\pi/2,\pi/2)$. If $c=0$,
assume that $x\notin 2\pi\ti\mathbb Z+\beta\mathbb R_{\leq 0}$. If $c>0$, assume that $x\notin \beta^{1-c}\mathbb R_{\leq 0}$. 
Then, as $\beta\rightarrow 0$,
$F(\beta)=\log(e^{-x\beta^c};e^{-\beta})_\infty$ is given by the following asymptotic series. If $c=0$,
\begin{subequations}\label{asymptotics-four-regimes}
\begin{equation}\label{asymptotics-fixed-y}F(\beta)
    \sim 
    \sum_{k=0}^{\infty}\frac{(-1)^{k-1}B_{k}\Li_{2-k}(e^{-x})}{k!}\,\beta^{k-1}.
    \end{equation}
    If $0<c<1$, 
    \begin{multline}\label{asymptotics-small-c}
F(\beta)\sim -\frac{\pi^2}{6}\,\beta^{-1}
+x\left(1-\log x-c\log\beta\right)\beta^{c-1}+\frac c2\,\log\beta+\frac 12\log x
\\
-\sum_{k=1}^\infty\frac{B_{2k}}{2k(2k-1)x^{2k-1}}\,\beta^{(1-c)(2k-1)}
-\sum_{n=1}^\infty\sum_{k=0}^{n+1}\frac{B_kB_nx^{n+1-k}}{k!n(n+1-k)!}\,\beta^{(1-c)(k-1)+cn}.
    \end{multline}
    If $c=1$,
\begin{equation}\label{asymptotics-c-1} F(\beta)\sim -\frac{\pi^2}{6}\,\beta^{-1}+\left(\frac 12-x\right)\log\beta-\log\Gamma(x)+\frac 12\log(2\pi)
-\sum_{n=1}^{\infty}\frac{B_nB_{n+1}(x)}{n(n+1)!}\,\beta^n,
\end{equation}
where
$B_n(x)=\sum_{k=0}^n\binom nk B_kx^{n-k}$
are  Bernoulli polynomials. Finally, if $c>1$,
\begin{multline}\label{asymptotics-large-c}
F(\beta)\sim-\frac{\pi^2}{6}\,\beta^{-1}+\left(c-\frac 12\right)\log\beta +\log x+\frac 12\log(2\pi)+x(\gamma-\log\beta)\beta^{c-1}\\
-\sum_{k=2}^\infty\frac{\zeta(k)(-x)^k}{k}
\beta^{(c-1)k}
-\sum_{n=1}^\infty\sum_{k=0}^{n+1}\frac{B_kB_nx^{n+1-k}}{k!n(n+1-k)!}\,\beta^{k-1+c(n+1-k)},
    \end{multline}
    \end{subequations}
    where $\zeta(k)$ are Riemann zeta values and $\gamma$ is Euler's constant. 
In \eqref{asymptotics-fixed-y}, $Li_1(e^{-x})$
    and $\Li_2(e^{-x})$ denote the branches that are
    real-valued for $x>0$ and analytic for
    $x\notin 2\pi\ti\mathbb Z+\beta\mathbb R_{\leq 0}$.
     In \eqref{asymptotics-small-c}--\eqref{asymptotics-large-c}, $\log x$ and $\log\Gamma(x)$ denote the branches that are
     real-valued for $x>0$ and analytic for 
  $x\notin \beta^{1-c}\mathbb R_{\leq 0}$.
\end{corollary}

To give examples, the first few terms in the cases $c=1/2$ and $c=2$ are
\begin{multline}
 \log(e^{-x\beta^{\frac 12}};e^{-\beta})_\infty=-\frac{\pi^2}{6}\,\beta^{-1}
+x\left(1-\log x-\frac{\log\beta}2\right)\beta^{-\frac 12}+
\frac{\log\beta}4
\\
+\frac{\log x}2+\frac{x^2}4-\left(\frac 1{12x}+\frac x4+\frac{x^3}{72}\right)\beta^{\frac 12}+\left(\frac 1{24}+\frac{x^2}{48}\right)\beta
+\mathcal O(\beta^{3/2}),
\end{multline}
\begin{multline}
    \log(e^{-x\beta^2};e^{-\beta})_\infty=-\frac{\pi^2}{6}\,\beta^{-1}+\left(\frac 32-x\beta\right)\log\beta +\log x+\frac{\log(2\pi)}2\\
   +\left(\frac 1{24}+\gamma x\right)\beta -\left(\frac x4+\frac{\pi^2x^2}{12}\right)\beta^2+\left(-\frac{x}{144}+\frac{x^2}4+\frac{\zeta(3)x^3}{3}\right)\beta^3
   +\mathcal O(\beta^4).\label{eq:xbeta2_asymp}
\end{multline}

\begin{proof}
In the case $c=0$, we start from
\eqref{log-product} (with $y$ replaced by $x$).
As in
the proof of Corollary \ref{corollary-uniform}, 
   if $x$ avoids a sector along each cut $2\pi\ti n+\beta\mathbb R_{\leq 0}$, we can estimate
$|f_N(x)|\leq C|x|^{-2N-1}$. 
 If  
 $x$ is fixed outside the cuts, this can  be achieved with the same angle chosen for each sector. We can then estimate
\begin{equation}\left|\sum_{n\in\mathbb Z}f_N\left(\frac{x+2\pi\ti n}{\beta}\right)\right|\leq
 C|\beta|^{2N+1} \sum_{n=-\infty}^\infty \frac 1{|x+2\pi \ti n|^{2N+1}}\leq D|\beta|^{2N+1},
 \end{equation}
 where $D$ depends on $x$, $N$ and $\arg\beta$ but is independent of $|\beta|$. This gives
 \eqref{asymptotics-fixed-y}.

 If $c>0$, we use Corollary \ref{corollary-uniform}
 in the form \eqref{log-product-estimate}, that is,
 $$\log(e^{-x\beta^c};e^{-\beta})_\infty=\sum_{k=0}^{2N}\frac{B_{k}\Li_{2-k}(e^{-x\beta^c})(-\beta)^{k-1}}{k!}-f_N(x\beta^{c-1})+\mathcal O(\beta^{2N+1}).$$
 Since 
 $x\beta^c\rightarrow 0$,  the only restriction on $x$ is  $x\notin\beta^{1-c}\mathbb R_{\leq 0} $.
By \eqref{li-exp}, 
\begin{equation}\label{polylog-taylor}\Li_{2-k}(e^{-x})=\sum_{n=\max(1,k-1)}^\infty \frac{(-1)^kB_nx^{n-k+1}}{n(n-k+1)!}+\begin{cases}\frac{\pi^2}6+x\log x-x, & k=0,\\
-\log x, & k=1\\
(k-2)!\,x^{1-k}, & k\geq 2,
\end{cases}\end{equation}
with convergence for $x$ close to $0$. The sum over $n$ gives a contribution
$$-\sum_{k=0}^{2N}\sum_{n=\max(1,k-1)}^{\infty}\frac{B_kB_nx^{n-k+1}\beta^{(1-c)(k-1)+cn}}{k!n(n-k+1)!}.$$
 Here, we can restrict to $n\leq P$ for
$P$ sufficiently large, so that all omitted terms 
are $\mathcal O(\beta^{2N+1})$. 
We can then ignore the condition $k\leq 2N$, since doing so
only adds finitely many terms that are all $\mathcal O(\beta^{2N+1})$.
Changing the order of summation, this gives
\begin{multline}\label{asymptotics-c-positive}
  \log(e^{-x\beta^c};e^{-\beta})_\infty=  -\frac{\pi^2}{6\beta}+\frac c2\,\log\beta
+x\beta^{c-1}\left(1-\log x-c\log\beta\right)+\frac 12\log x
\\-\sum_{k=1}^N \frac{B_{2k}\beta^{(1-c)(2k-1)}}{2k(2k-1)x^{2k-1}}
-\sum_{n=1}^{P}\sum_{k=0}^{n+1}\frac{B_kB_nx^{n-k+1}\beta^{(1-c)(k-1)+cn}}{k!n(n-k+1)!}\\
-f_N(x\beta^{c-1})+\mathcal O(\beta^{2N+1}).
\end{multline}

In the case $0<c<1$, let $M$ be large enough so that $(1-c)(2M+1)\geq 2N+1$. Then,
\begin{align*}f_N(x\beta^{c-1})&=f_M(x\beta^{c-1})+\sum_{k=N+1}^M \frac{B_{2k}\beta^{(1-c)(2k-1)}}{2k(2k-1)x^{2k-1}}\\
&=\sum_{k=N+1}^M \frac{B_{2k}\beta^{(1-c)(2k-1)}}{2k(2k-1)x^{2k-1}}+\mathcal O(\beta^{2N+1}). 
\end{align*}
It follows that 
\begin{multline*}
  \log(e^{-x\beta^c};e^{-\beta})_\infty=  -\frac{\pi^2}{6\beta}+\frac c2\,\log\beta
+x\beta^{c-1}\left(1-\log x-c\log\beta\right)+\frac 12\log x
\\-\sum_{k=1}^M \frac{B_{2k}\beta^{(1-c)(2k-1)}}{2k(2k-1)x^{2k-1}}
-\sum_{n=1}^{P}\sum_{k=0}^{n+1} \frac{B_kB_nx^{n-k+1}\beta^{(1-c)(k-1)+cn}}{k!n(n-k+1)!}
+\mathcal O(\beta^{2N+1})
\end{multline*}
for all sufficiently large $M$ and $P$, 
which is an equivalent way of stating \eqref{asymptotics-small-c}.

In the case $c=1$, we simply insert the explicit expression for $f_N(x)$ from \eqref{stirling-general} into \eqref{asymptotics-c-positive}.  This gives \eqref{asymptotics-c-1} after simplification.

Finally, if $c>1$ we again insert \eqref{stirling-general} into \eqref{asymptotics-c-positive} and use  \cite[Eq.\ (1.17.2)]{bateman}
$$\log\Gamma(x)=-\log x-\gamma x+\sum_{k=2}^\infty\frac{\zeta(k)(-x)^k}{k}. $$


\end{proof}

Ramanujan stated the first five non-zero terms ($k=0,1,2,4,6$) in \eqref{asymptotics-fixed-y} explicitly
and also gave the general case of  \eqref{asymptotics-c-1},
see \cite[\S 27, Entry 6 and Entry 6$'$]{berndt}.  
Berndt comments that ``the form of the asymptotic expansion in Entry 6$'$ \dots 
is much inferior to that in Entry 6". This comparison does not seem fair since the two results
give the complete asymptotic expansions in two different regimes. Moak
\cite[Thm.\ 3]{moak}
 also obtained a version of 
   \eqref{asymptotics-fixed-y}, with the numerator in $\Li_{2-k}$ for $k\geq 2$ (see \eqref{polylog-rational}) defined recursively. In a form more similar to 
ours, \eqref{asymptotics-fixed-y} and \eqref{asymptotics-c-1} were obtained by 
 McIntosh \cite[Eq.\ (6.4)]{McIntosh95}, \cite[Thm.\ 2 and Thm.\ 4]{McIntosh99},
who also more generally considered the asymptotics of $(e^{-\alpha-x\beta};e^{-\beta})_\infty$. 
 For some more recent related work, we mention
 \cite{zagier}, which gives the complete asymptotics of $(aq^x;q)_\infty$
 as $q$ tends to a root of unity, and 
 \cite{fantini-rella}, which discusses resurgence properties of the expansions \eqref{asymptotics-fixed-y} and \eqref{asymptotics-c-1}.

We finish this section with a comment on another interesting limit of $(e^{-y};e^{-\beta})_\infty$,
namely,  $\Re y\to-\infty$ for fixed $\beta$ and $\Im y$. This limit was discussed in \cite[\S 3.3]{daalhuis1994asymptotic} for $\Im\beta=0$. 
Starting from \eqref{main-identity-convergent}, we note that the argument of $\Gamma\big((y+2\pi i n)/\beta\big)$ leaves every sector of validity of  Stirling's expansion. Therefore, to find the asymptotics of $(e^{-y};e^{-\beta})_\infty$, we must first apply Euler's reflection formula
\begin{equation*}
    \frac{1}{\Gamma(z_n)} = \frac{\sin(\pi z_n)}{\pi} \Gamma(1 - z_n) = \frac{\sin(\pi z_n)}{\pi} (-z_n) \Gamma(-z_n),
\end{equation*}
and then use the Stirling expansion for $\Gamma(-z_n).$ A straightforward computation, using also the reflection formula for the dilogarithm,  leads to
\begin{equation}\label{eq:large-y}
\begin{split}
(e^{-y}; e^{-\beta})_\infty &\sim \exp\left( \frac{y^2}{2\beta} - \frac{y}{2} + \frac{\beta}{12} - \frac{\pi^2}{3\beta} \right)\\
&\ \ \times2 \sin\left(\frac{\pi y}{\beta}\right) \prod_{n=1}^\infty \left(1 - e^{\frac{2\pi i y}{\beta}} e^{-\frac{4\pi^2 n}{\beta}}\right) \left(1 - e^{-\frac{2\pi i y}{\beta}} e^{-\frac{4\pi^2 n}{\beta}}\right),
\end{split}
\end{equation}
with  $\mathcal{O}(e^{\Re y})$ relative error. We have suppressed the details of the derivation since \eqref{eq:large-y} is not new. It follows from the  modular transformation for Jacobi theta functions, 
that is, the case $y=-x\beta$
of \eqref{theta-modular}. We have discarded the factor $(e^{y-\beta};e^{-\beta})_\infty$ which is exponentially close to $1$ as $\Re y\to-\infty$. For $\Im\beta=0$, the product on the right-hand side of \eqref{eq:large-y} is $\mathcal{O}(y^0)$, and hence \eqref{eq:large-y} gives the desired  asymptotics, compatible with \cite{daalhuis1994asymptotic}.

\section{Error analysis}\label{sec:numerics}

In this section, we study the numerical accuracy of  the asymptotic series in \eqref{uniform-asymptotics} and \eqref{asymptotics-four-regimes}. 
This will be based on the heuristics that, for an asymptotic series
$$f\sim\sum_{N=0}^\infty T_N,$$
the optimal truncation error 
$$R_\ast=\left|f-\sum_{N=0}^{N_\ast} T_N\right|$$
typically appears when $|T_{N_\ast}|\approx |T_{N_\ast+1}|$ and that then
$R_\ast\lesssim |T_{N_\ast}|$.
We  make no attempts to prove such results rigorously, but
we  verify that the heuristic estimates agree rather well with numerical computations. These  computations were performed in Mathematica using high-precision arithmetic. 

Most cases will be treated by comparison with the terms
$$T_N=C(N-1)!\,t^N.$$
Since
${T_{N+1}}/{T_N}=Nt$,
the heuristic estimates give 
\begin{subequations}
   \label{heuristic-comparison}
\begin{align}
N_\ast&\approx 1/t,\\
R_\ast&\lesssim |T_{N_\ast}|\approx \frac{|C|(N_\ast-1)!}{N_\ast^{N_\ast}}\approx \frac{|C|\sqrt{2\pi}e^{-N_\ast}}{\sqrt{N_\ast}}\approx |C|\sqrt{2\pi t}\,e^{-1/t}, \end{align} 
\end{subequations}
where we  used Stirling's approximation of the factorial. 

 We start with
 the uniform asymptotic expansion in \eqref{uniform-asymptotics}, where we replace $y$ by $x$. The $k$-th term is 
\[
T_k
=
\frac{B_{2k}\,\beta^{2k-1}}{(2k)!}
\left(
\operatorname{Li}_{2-2k}(e^{- x})
-
\frac{(2k-2)!}{{x}^{2k-1}}
\right).
\]
For simplicity, we take $\beta>0$ and 
$ x>0$. It follows from \eqref{polylog-parfrac} that
\begin{align*}\Li_{2-2k}(e^{- x})-
\frac{(2k-2)!}{{x}^{2k-1}}
&=(2k-2)!\sum_{n=1}^\infty\left(\frac 1{(x+2\pi\ti n)^{2k-1}}+\frac 1{(x-2\pi\ti n)^{2k-1}}\right)\\
&=\sum_{n=1}^\infty\frac{2(2k-2)!\cos\big((2k-1)\arctan(2\pi n/x)\big)}{(x^2+4\pi^2n^2)^{(2k-1)/2}}.
\end{align*}
Let $\theta=\arctan(2\pi/x)$.
If we let $k\rightarrow\infty$ along a sequence where $ \cos\big((2k-1)\theta\big)$ is bounded away from zero, then the $n=1$ term dominates and we have
$$\Li_{2-2k}(e^{- x})-
\frac{(2k-2)!}{{x}^{2k-1}}\sim\frac{2(2k-2)!\cos\big((2k-1)\theta\big)}{(x^2+4\pi^2)^{(2k-1)/2}}. $$
Using also  the asymptotics \cite[Eq.\ (1.13.22)]{bateman}
\begin{equation}
B_{2k}= (-1)^{k-1}\frac{2(2k)!\,\zeta(2k)}{(2\pi)^{2k}}\sim (-1)^{k-1}\frac{2(2k)!}{(2\pi)^{2k}},\label{eq:Bernoulli_n-asymptotics-large-n}
\end{equation}
we arrive at
\begin{equation}
T_k
\sim
(-1)^{k-1}\!\cos\big((2k-1)\theta\big) \frac{4(2k-2)!}{(2\pi)^{2k}}
\frac{\beta^{2k-1}}{({x}^2+4\pi^2)^{(2k-1)/2}}.\label{eq:Tk-estimate}
\end{equation}
In particular, $|T_k|\rightarrow\infty$, so the asymptotic series \eqref{uniform-asymptotics} is divergent. 

We will now estimate the optimal truncation error. 
Ignoring the sign and the cosine factor 
in \eqref{eq:Tk-estimate}
(we will comment on their effect below), 
the terms behave as
$$\frac{4(N-1)!\,\beta^N}{(2\pi)^{N+1}(x^2+4\pi^2)^{N/2}}. $$
Applying \eqref{heuristic-comparison}, with
$C=2/\pi$ and $t=\beta/2\pi\sqrt{x^2+4\pi^2}$, gives 
\begin{equation}\label{eq:uniform-kast-Tkast}N_\ast\approx \frac{2\pi\sqrt{x^2+4\pi^2}}\beta,
\qquad R_\ast\lesssim \frac{2\sqrt\beta e^{-2\pi\sqrt{x^2+4\pi^2}/\beta}}{\pi\sqrt[4]{x^2+4\pi^2}}
\qquad \text{for \eqref{uniform-asymptotics}.}
\end{equation}
 In regimes where  $x\approx 0$, this simplifies  to 
\begin{equation}N_\ast\approx \frac{4 \pi^2}{\beta},\qquad R_\ast\lesssim\frac{\sqrt{2\beta} e^{-4\pi^2/\beta}}{\pi^{3/2}}.
\label{eq:uniform-error-x=0}\end{equation}


Figure~\ref{fig:uniform-error} illustrates the error of the partial sums in \eqref{uniform-asymptotics} as a function of the truncation order
(largest exponent of $\beta$ used in the expansion). The estimates for $N_\ast$ and $R_\ast$
given in \eqref{eq:uniform-kast-Tkast} are indicated by dashed lines.
To compare with the alternative expansions \eqref{asymptotics-four-regimes} (see Figure \ref{fig:asymptotic-error}), we
consider the four cases  $y=x\beta^c$ where $c=0,\,1/2,\, 1,\,2$, with $x=3$ and $\beta=1/16$ fixed. 
When $c=2$
 there is a sharp transition from
better to worse approximation
occurring at order
$643$.
For $c=1$  we observe a different behaviour. This results from the factor $F:=(-1)^{k-1}\cos((2k-1)\theta)$ in \eqref{eq:Tk-estimate},
which was ignored in the analysis above. In the case $c=2$, $2\theta\approx 3.138$
is very near $\pi$, which means that $F$ seldom changes sign. In particular, it has constant sign near the optimal truncation point, which means that 
 the error decreases until the exact value is crossed and increases from there. 
 When $c=1$, $2\theta\approx 3.082$, so $F$ oscillates faster. In particular, it changes sign  close to the estimated optimal truncation point.
In this case the error starts building up before reaching the smallest term in modulus. When the terms change sign the error starts decreasing, before increasing again when the sum crosses the exact value for the second time. 
For smaller $c$  we observe even faster oscillations around the exact value. As $c$ decreases, the optimal truncation order  increases,
which leads to a large decrease in the truncation error.

\begin{figure}[ht]
    \centering
    \begin{subfigure}{0.48\textwidth}
    \centering
\includegraphics[width=\textwidth]{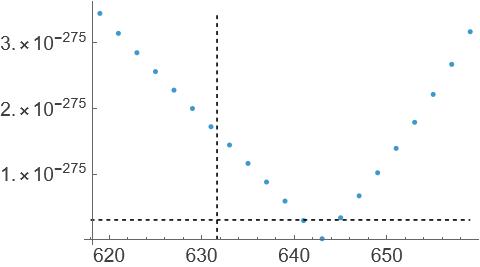}
\caption{$c=2$}
\end{subfigure}%
\hfill
 \begin{subfigure}{0.48\textwidth}
    \centering
     \includegraphics [width=\textwidth]{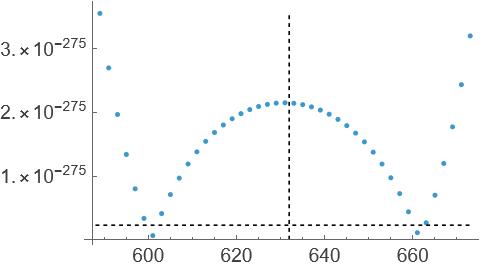}
  \caption{$c=1$}
\end{subfigure}   
     \\
     \vspace{5mm}
      \begin{subfigure}{0.48\textwidth}
    \centering
    \includegraphics[width=\textwidth]{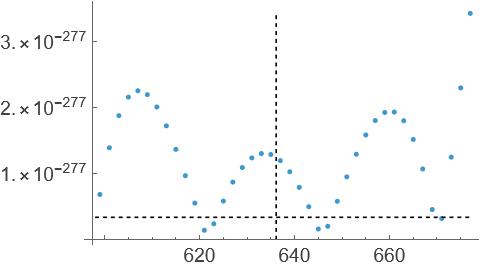}
    \caption{$c=0.5$}
\end{subfigure}
    \hfill
    \begin{subfigure}{0.48\textwidth}
    \centering
     \includegraphics[width=\textwidth]{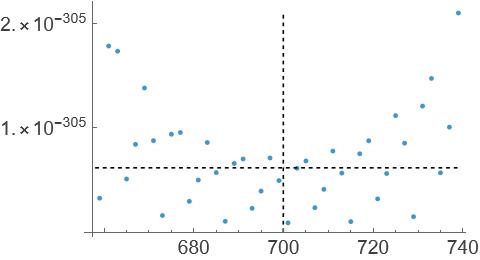}
    \caption{$c=0$}
\end{subfigure}   
    \caption{Error of the partial sums of \eqref{uniform-asymptotics} as a function of the truncation order. We take 
     $y=x\beta^c$, with  $x=3$, $\beta=1/16$ and the four values $c=0,\,1/2,\, 1,\, 2$.
     The dashed lines are the estimates for the optimal truncation order $N_\ast$ and truncation error $R_\ast$ computed from \eqref{eq:uniform-kast-Tkast}.}
    \label{fig:uniform-error}
\end{figure}




We now turn to the complete asymptotic expansions \eqref{asymptotics-four-regimes}, still assuming $x,\beta>0$.
In \eqref{asymptotics-fixed-y},  the terms behave as 
$$
\frac{B_{N+1}\,\beta^{N}}{(N+1)!}
\operatorname{Li}_{1-N}(e^{- x})\sim 
(-1)^{(N-1)/2} \frac{2(N-1)!}{(2\pi)^{N+1}}\left(\frac \beta x\right)^{N}.
$$
Using \eqref{heuristic-comparison},
with $C=1/\pi$ and $t=\beta/2\pi x$, gives the estimates
\begin{equation}
N_\ast \approx  \frac{2\pi x}{\beta},
 \qquad R_\ast\lesssim\frac{\sqrt{\beta}e^{-2\pi x/\beta}}{\pi\sqrt{x}},
\qquad\text{for \eqref{asymptotics-fixed-y}.}
\label{eq:kast-fixed-y}
\end{equation}
Compared to \eqref{eq:uniform-kast-Tkast},
the optimal truncation appears earlier in the expansion but has a larger error. The reason is that, in \eqref{asymptotics-fixed-y}, the main contribution to 
$R_\ast$ comes from replacing  the term $\log\Gamma(x/\beta)$
 with
Stirling's asymptotic series. In 
\eqref{uniform-asymptotics}, this term is assumed to be known exactly.

    




Turning to the series \eqref{asymptotics-small-c}, we claim that the main error comes from the single series. 
Since $|B_n|<4n!/(2\pi)^n$ (a consequence of \eqref{eq:Bernoulli_n-asymptotics-large-n}), the terms in the double series can be estimated by
\begin{equation}\label{double-sum-estimate}\frac{16(n-1)!\,x^{n+1-k}}{(2\pi)^{k+n}(n+1-k)!}\,\beta^{(1-c)(k-1)+cn}.\end{equation}
If the exponent of $\beta$ is $N$, then $(k,n)$ 
is on the line segment
$(1-c)(k-1)+cn=N$ where
$0\leq k\leq n+1$. It follows that $k\leq N+1$ and $n\leq (N+1-c)/c $. We can then estimate
\begin{equation}\label{dominant-factor-estimate}\frac{(n-1)!}{(n+1-k)!}\leq (n-1)^{k-2}\leq \left(\frac{N+1-2c}c\right)^{N-1}.\end{equation}
Hence,  each term can be estimated by $(N\beta)^{N}$ times subdominant corrections. If we allow the exponent 
of $\beta$ to vary in some interval near $N$, then $\mathcal O(N)$ terms contribute, so the total contribution is still roughly of size  $(N\beta)^N$. 
On the other hand, it follows from \eqref{eq:Bernoulli_n-asymptotics-large-n} that the terms in the single sum  behave as
$N^{N/(1-c)}\beta^N$ times subdominant factors. Hence, for large $N$, the contribution of the single sum dominates that of the double sum. 
We can write the  terms in the single sum  as
$$\frac{B_{N+1}\beta^{(1-c)N}}{N(N+1)x^N}\sim \frac{(N-1)!\,\beta^{(1-c)N}}{\pi^{N+1}(2 x)^N}.$$
We apply \eqref{heuristic-comparison}
with $C=1/\pi$ and  $t=\beta^{1-c}/2\pi x$, where 
we must multiply the expression for 
$N_\ast$ by $1-c$. This gives the estimates
\begin{equation}
    N_\ast \approx 2(1-c)\pi x\beta^{c-1},
 \qquad R_\ast\lesssim \frac{{\beta^{\frac{1-c}2}}e^{-2\pi x\beta^{c-1}}}{\pi\sqrt{x}},
\qquad\text{for \eqref{asymptotics-small-c}}.\label{eq:kast-small-c}
\end{equation}
Just as for \eqref{asymptotics-fixed-y}, the main error comes from using
Stirling's asymptotic series for the term
$\log\Gamma(x\beta^{c-1})$.

    

We now turn to \eqref{asymptotics-c-1}. 
The  terms are
$$T_n=\frac{B_{n}B_{n+1}(x)}{n(n+1)!}\,\beta^{n},
$$
where $n$ is even unless $n=1$.
The asymptotic behaviour is  different depending on whether $x\in\mathbb Z/2$ or not. An easy way to understand this is from
the Fourier expansion
\begin{equation}\label{bernoulli-fourier}B_{n}(x)= n\sum_{k=1}^{\lfloor x\rfloor}(x-k)^{n-1}
+2(-1)^{(n+1)/2}n!\sum_{k=1}^\infty \frac{\sin(2\pi k x)}{(2\pi k)^{n}},
\end{equation}
which holds for $x\geq 0$ and $n\geq 3$ odd. This 
 is obtained by combining \cite[Eqs.\ (1.13.6) and (1.13.15)]{bateman}.

Assume first that
$x\notin\mathbb Z/2$.
Then, for large $n$ \eqref{bernoulli-fourier}
is dominated by the first term in the infinite series, which gives (for $n\geq 2$ even) 
\begin{equation}
    B_{n+1}(x)\sim(-1)^{n/2+1}\frac{2(n+1)!}{(2\pi)^{n+1}}\sin\left(2\pi x\right),\qquad x\notin\mathbb Z/2.\label{eq:Bern-pol-asy}
\end{equation}
Using also \eqref{eq:Bernoulli_n-asymptotics-large-n}, we find  that 
$$T_n\sim \frac{4(n-1)!}{(2\pi)^{2n+1}}\,\sin(2\pi x)\beta^{n}.
$$
Applying \eqref{heuristic-comparison}, with $C=2\sin(2\pi x)/\pi$ and $t=\beta/4\pi^2$, leads to 
\begin{equation}\label{eq:kast-c=1}
N_\ast\approx\frac{4\pi^2}{\beta},\qquad R_\ast\lesssim\frac{\sqrt{2\beta}|\sin(2\pi x)| e^{-4\pi^2/\beta}}{\pi^{3/2}}\qquad \text{for \eqref{asymptotics-c-1} if\ }  x\notin\mathbb Z/2,
\end{equation}
which is similar to 
\eqref{eq:uniform-error-x=0}.

    

In the case  $x\in\mathbb Z/2$, 
 \eqref{bernoulli-fourier} reduces to the classical identities
$$B_{n}(x)=\begin{cases}
n\sum_{k=1}^{x-1}k^{n-1},&  x\in\mathbb Z_{\geq 0},\\
n\sum_{k=1}^{x-1/2}\left(k-\frac 12\right)^{n-1},&
x\in\frac 12+ \mathbb Z_{\geq 0}.
\end{cases}$$
This holds for odd $n\geq 3$. 
For $n=2$, one has the modified identities
$$B_{2}(x)=\begin{cases}
\frac 16+2\sum_{k=1}^{x-1}k,&  x\in\mathbb Z_{\geq 0},\\
-\frac 1{12}+2\sum_{k=1}^{x-1/2}\left(k-\frac 12\right),&
x\in\frac 12+ \mathbb Z_{\geq 0}.
\end{cases}$$
Using also \eqref{li1-exp}, we find that for positive integer $x$ and $\beta$ close to $0$,
\begin{align}\nonumber
\sum_{n=1}^\infty \frac{B_nB_{n+1}(x)\beta^n}{n(n+1)!}&=-\frac\beta{24}+
\sum_{k=1}^{x-1}\sum_{n=1}^\infty\frac{B_n(k\beta)^n}{n\, n!}
=-\frac{\beta}{24}-\sum_{k=1}^{x-1}\log\frac{\beta k}{1-e^{-\beta k}}\\
\label{conv-series}&=-\frac\beta{24}-\log\frac{\beta^{x-1}(x-1)!(e^{-x\beta};e^{-\beta})_\infty}{(e^{-\beta};e^{-\beta})_\infty}.
\end{align}
In particular, the asymptotic series in 
\eqref{asymptotics-c-1} is  convergent. However, \eqref{asymptotics-c-1}
does not hold as an identity. In fact,
combining \eqref{conv-series} with \eqref{dedekind} gives 
\begin{multline*}
\log(e^{-x\beta};e^{-\beta})_\infty
=-\frac{\pi^2}{6}\,\beta^{-1}+\left(\frac 12-x\right)\log\beta-\log\Gamma(x)+\frac 12\log(2\pi)\\
-\sum_{n=1}^{\infty}\frac{B_nB_{n+1}(x)}{n(n+1)!}\,\beta^n
+R(\beta),
\end{multline*}
with the exact remainder term
\begin{equation}\label{exact-remainder}R(\beta)=\log(e^{-4\pi^2/\beta};e^{-4\pi^2/\beta})_\infty.\end{equation}
When $x$ is a half-integer, a similar computation gives
$$R(\beta)=\log(-e^{-4\pi^2/\beta};e^{-4\pi^2/\beta})_\infty. $$

    

Finally, we turn to \eqref{asymptotics-large-c}.
From numerical experiments, 
it appears that the 
series behaves very similarly to
the uniform expansion \eqref{uniform-asymptotics}.
Hence, by
\eqref{eq:uniform-error-x=0}, we expect that
\begin{equation}\label{eq:kast-c>1}N_\ast\approx \frac{4\pi^2}{\beta},\qquad 
R_\ast\lesssim \frac{\sqrt{2\beta} e^{-4\pi^2/\beta}}{\pi^{3/2}}
 \qquad 
\text{for \eqref{asymptotics-large-c}}.
\end{equation}
We will give an informal argument for this. 
As in the case of \eqref{asymptotics-small-c}, we
 expect that
the double sum is dominated by terms with  $n$ and $k$ both large. In particular, they are both even. If  $n=k+2j$,
then for large $n$ and fixed $j$
the terms behave as
$$\frac{(2\pi x\beta^{c-1})^{2j+1}}{(2j+1)!}\cdot \frac{4(n-1)!\beta^n}{(2\pi)^{2n+1}}.$$
Here, the second factor behaves as
\eqref{eq:kast-c>1}, whereas the first factor
is at most $\mathcal O(\beta^{c-1})$ (when $j=0$).
This indicates that the truncation error from the double sum is dominated by the error
\eqref{eq:kast-c>1} that was already present before we used the expansion \eqref{polylog-taylor}. The terms in the single series 
 behave as  $x^{N/(c-1)}\beta^N/N$ and are negligible. 

To summarize the discussion of 
\eqref{asymptotics-four-regimes}, if $c>1$, or $c=1$ and $x\notin\mathbb Z/2$, the optimal truncation error is essentially the same as in \eqref{uniform-asymptotics}, which comes from the last sum 
in \eqref{uniform-decomposition}. For $0\leq c<1$ the main error  comes from expanding the term
$\log\Gamma(x\beta^{c-1})$ as an asymptotic series. In the exceptional case $c=1$, $x\in\mathbb Z/2$, the asymptotic series is convergent, but there is still a non-trivial and explicitly known error
of size $\mathcal O(e^{-4\pi^2/\beta})$.

Figure \ref{fig:asymptotic-error}
illustrates the different cases. When $c=2$, the behaviour is similar to the
 uniform expansion \eqref{uniform-asymptotics} illustrated in Figure
\ref{fig:uniform-error}.
 The main  difference is that 
 \eqref{uniform-asymptotics}
 only contains odd powers of $\beta$, whereas 
\eqref{asymptotics-large-c} (with $c=2$)
also contains even powers. 
However, the error is dominated by terms in the double sum with $k\approx n\approx N_\ast$, which correspond to odd exponents. 
Hence, as can be seen in the figure, the terms with even exponents are negligible. The case $c=1$ illustrates the exceptional situation when the asymptotic series converges. In the two cases $c=0.5$
and $c=0$ the picture essentially illustrates the error in Stirling's asymptotic series. For $c=0.5$ we also see the effect of  negligible terms with half-integer exponents. 

\begin{figure}[ht]
    \centering
    \begin{subfigure}{0.48\textwidth}
    \centering
\includegraphics[width=\textwidth]{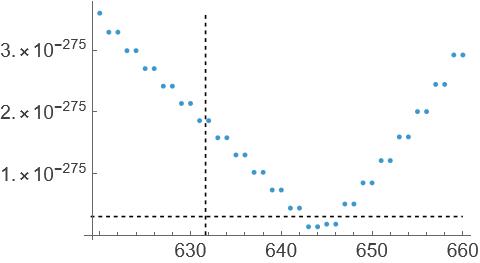}
\caption{$c=2$}
\end{subfigure}%
\hfill
 \begin{subfigure}{0.48\textwidth}
    \centering
     \includegraphics [width=\textwidth]{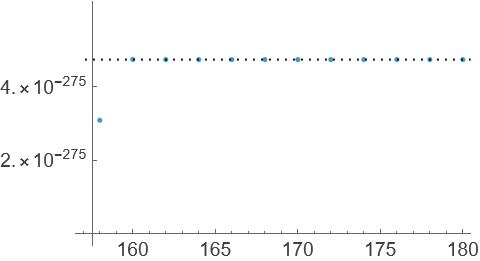}
  \caption{$c=1$}
\end{subfigure}   
     \\
     \vspace{5mm}
      \begin{subfigure}{0.48\textwidth}
    \centering
    \includegraphics[width=\textwidth]{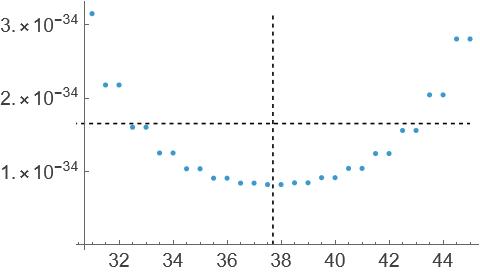}
    \caption{$c=0.5$}
\end{subfigure}
    \hfill
    \begin{subfigure}{0.48\textwidth}
    \centering
     \includegraphics[width=\textwidth]{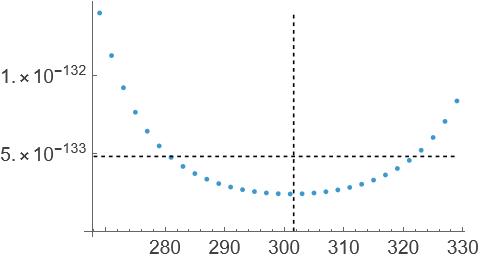}
    \caption{$c=0$}
\end{subfigure}   
    \caption{Error of the partial sums of \eqref{asymptotics-four-regimes} as a function of the truncation order. We take 
     $y=x\beta^c$, with  $x=3$, $\beta=1/16$ and  $c=0,\,1/2,\, 1,\, 2$.
For the cases $c\neq 1$,
     the dashed lines are the estimates for the optimal truncation order $N_\ast$ and truncation error $R_\ast$ computed from 
     \eqref{eq:kast-fixed-y}, \eqref{eq:kast-small-c} and
     \eqref{eq:kast-c>1}.
     For  $c=1$, 
     the dotted line indicates the exact remainder term \eqref{exact-remainder}.}
    \label{fig:asymptotic-error}
\end{figure}

In Figure \ref{fig:generic},
we illustrate the error in a case when
$c=1$ and $x\notin\mathbb Z/2$.
We see that the uniform and the complete asymptotic expansion have a 
similar, but not identical, behaviour.

\begin{figure}[ht]
    \centering
    \begin{subfigure}{0.48\textwidth}
    \centering
\includegraphics[width=\textwidth]{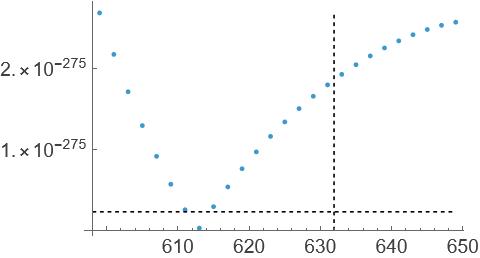}
\end{subfigure}%
\hfill
 \begin{subfigure}{0.48\textwidth}
    \centering
     \includegraphics [width=\textwidth]{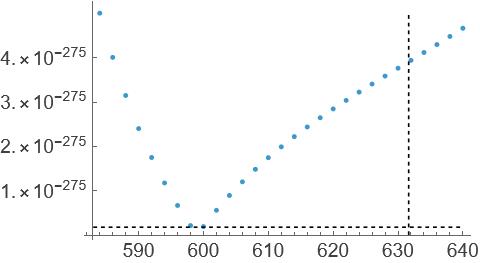}
\end{subfigure}   
    \caption{Truncation errors 
    in the case $y=x\beta$, with $\beta=1/16$ and 
    $x=2.9$. To the left, the error in the uniform expansion \eqref{uniform-asymptotics}, with estimates for the optional truncation order and error from \eqref{eq:uniform-kast-Tkast}. To the right, the error in the complete asymptotic expansion
    \eqref{asymptotics-c-1}, with estimates from \eqref{eq:kast-c=1}.
    }
    \label{fig:generic}
\end{figure}

\appendix

\section{Polylogarithms.}

For convenience of reference, we collect some classical results on the polylogarithms
$$\Li_n(x)=\sum_{k=1}^\infty\frac{x^k}{k^n},$$
see e.g.\ \cite[\S 1.11]{bateman}.
We will only encounter the case when $n\leq 2$ is an integer. 
If  $n\leq 1$, $\Li_k$ is an elementary function. Namely,
$$\Li_0(x)=\frac x{1-x},\qquad \Li_1(x)=\log \frac 1{1-x},$$
\begin{equation}\label{polylog-rational}\Li_{-n}(x)=\frac{\sum_{k=0}^{n-1}\euler nk x^{n-k}}{(1-x)^{n+1}},\qquad n\geq 1,\end{equation}
 where $\euler nk$ are Eulerian numbers. 

We need the power series expansions of $\Li_n(e^{-x})$. We start from
$$\Li_0(e^{-x})=\frac 1{e^x-1}=\sum_{k=0}^\infty\frac{B_kx^{k-1}}{k!},$$
where $B_k$ are the Bernoulli numbers. Since $\Li_n(e^{-x})'=-\Li_{n-1}(e^{-x})$, we can differentiate and integrate this identity to obtain
\begin{subequations}\label{li-exp}
\begin{align}\Li_{-n}(e^{-x})
&=\frac{n!}{x^{n+1}}+(-1)^n\sum_{k=0}^{\infty}\frac{B_{n+k+1}x^k}{(n+k+1)k!},\qquad n\geq 0,\\
\label{li1-exp}\Li_1(e^{-x})&=-\log x-\sum_{k=1}^\infty \frac{B_kx^k}{k\, k!},\\
\Li_2(e^{-x})&=\frac{\pi^2}6+x\log x-x+\sum_{k=1}^\infty \frac{B_kx^{k+1}}{k(k+1)!}.
\end{align}
\end{subequations}

Finally, we note that differentiating the classical partial fraction expansion
\begin{equation}\label{coth-parfrac}\operatorname{PV}\sum_{k\in\mathbb Z}\frac 1{x+2\pi\ti k}=
\frac{\coth(x/2)}2=
\frac 12 +\Li_0(e^{-x})\end{equation}
gives
\begin{equation}\label{polylog-parfrac}\sum_{k\in\mathbb Z}\frac 1{(x+2\pi\ti k)^{n+1}}=\frac 1{n!}\Li_{-n}(e^{-x}),\qquad n\geq 1.\end{equation}

\section{Narukawa's identity.}
\label{app:narukawa}

We will sketch how to  obtain \eqref{main-identity-convergent} as a formal limit of Narukawa's identity. 
We will use the notation  of \cite{Rains:2006dfy}.
If $\Im(\omega_1)>0$, $\Im(\omega_2)>0$ and $0<\Im(x)<\Im(\omega_1+\omega_2)$, the 
 hyperbolic  gamma function is defined by
$$\log \Gamma_h(x;\omega_1,\omega_2)=\ti \int_0^\infty \left(\frac{2x-\omega_1-\omega_2}{2t\omega_1\omega_2}-\frac{\sin(t(2x-\omega_1-\omega_2))}{2\sin(\omega_1 t)\sin(\omega_2t)}\right)\frac{dt}t.$$
Assuming also that  $\Im(\omega_1/\omega_2)>0$,
there are the alternative expressions \cite{Shintani}
\begin{multline}\label{shintani}\log\Gamma_h(x;\omega_1,\omega_2) \\
=\pi\ti P(x;\omega_1,\omega_2)+\log(e^{-2\pi\ti(x-\omega_1)/\omega_2};e^{2\pi\ti\omega_1/\omega_2})_\infty-\log(e^{-2\pi\ti x/\omega_1};e^{-2\pi\ti\omega_2/\omega_1})_\infty\\
=-\pi\ti P(x;\omega_1,\omega_2)-\log(e^{2\pi\ti x/\omega_2};e^{2\pi\ti\omega_1/\omega_2})_\infty+\log(e^{2\pi\ti(x-\omega_2) /\omega_1};e^{-2\pi\ti\omega_2/\omega_1})_\infty,
\end{multline}
where
$$P(x;\omega_1,\omega_2)=\frac{3(2x-\omega_1-\omega_2)^2-\omega_1^2-\omega_2^2}{24\omega_1\omega_2}. $$
The two expressions in \eqref{shintani}
are related by the  symmetry
$\log\Gamma_h(\omega_1+\omega_2-x)=-\log\Gamma_h(x)$ and also follow from each other using \eqref{theta-modular}.

If $|p|<1$ and $|q|<1$, the elliptic gamma function is defined by
\begin{equation}
    \Gamma_e(z;p,q)=\prod_{j,k=0}^\infty\frac{1-p^{j+1}q^{k+1}/z}{1-p^jq^kz}.\label{eq:elliptic_gamma_def}
\end{equation}
In the region $|pq|<|z|<1$, we choose the logarithm
$$\log\Gamma_e(z;p,q)=\sum_{n=1}^\infty \frac{z^n-(pq/z)^n}{n(1-p^n)(1-q^n)}.$$
Narukawa's identity can then 
 be written
\begin{multline}\label{narukawa}\sum_{n\in\mathbb Z}\log\Gamma_h(x+n;\omega_1,\omega_2)\pm \pi\ti P(x+n;\omega_1,\omega_2)\\
= \log \Gamma_e(e^{2\pi\ti x};e^{2\pi\ti\omega_1},e^{2\pi\ti\omega_2})-2\pi\ti Q(x;\omega_1,\omega_2),
 \end{multline}
where one should take the plus sign for $n\leq 0$ and the minus sign for $n>0$,
and where
$$Q(x;\omega_1,\omega_2)=\frac{(2x-\omega_1-\omega_2+1)(-(2x-\omega_1-\omega_2+1)^2+\omega_1^2+\omega_2^2+1)}{48\omega_1\omega_2}. $$
If one inserts the expressions \eqref{shintani} in
\eqref{narukawa} and take the exponential of both sides, one obtains after simplification 
a modular transformation for the elliptic gamma function  due to Felder and Varchenko
\cite[Thm.\ 4.1]{felder-varchenko},
\begin{multline*}e^{-2\pi\ti Q(x;\omega_1,\omega_2)}\Gamma_e(e^{2\pi\ti x};e^{2\pi\ti\omega_1},e^{2\pi\ti\omega_2})\\
=\Gamma_e(e^{-2\pi\ti (x+1)/\omega_1};e^{-2\pi\ti\omega_2/\omega_1},e^{-2\pi\ti/\omega_1})\Gamma_e(e^{2\pi\ti x/\omega_2};e^{2\pi\ti\omega_1/\omega_2},e^{-2\pi\ti/\omega_2}).
\end{multline*}
That is, in the generic case  $\omega_1/\omega_2\notin\mathbb R$, Narukawa's identity is  equivalent to 
the Felder--Varchenko transformation. 

It is clear from \eqref{q-pochhammer} and \eqref{eq:elliptic_gamma_def} that
\begin{equation}
    \Gamma_e(z;0,q)=\frac 1{(z;q)_\infty}.
\end{equation}
This can be given a heuristic QFT interpretation as follows. The elliptic gamma function $\Gamma_e(z;p,q)$ is, up to a Casimir-type factor, the BPS partition function of a 4d $\mathcal{N}=1$ chiral multiplet on a primary Hopf surface $\mathcal{M}_{p,q}$ \cite{Assel:2014paa} with complex-structure moduli $p,q$. The Hopf surface $\mathcal{M}_{p,q}$ is the complex surface obtained from $\mathbb{C}^2\setminus\{0\}$ by the identification $(z_1,z_2)\sim(p z_1,q z_2)$. It is diffeomorphic to $S^3\times S^1$, as can be seen from the parametrization (see e.g.~\cite{Closset:2013vra})
\begin{equation*}
    z_1=p^x\cos\frac{\theta}{2}\,e^{i\varphi},\quad z_2=q^x\sin\frac{\theta}{2}\,e^{i\chi}.
\end{equation*}
Here $x\sim x+1$ parametrizes the circle, while $\theta\in[0,\pi]$, $\varphi\sim\varphi+2\pi$, $\chi\sim\chi+2\pi$ parametrize $S^3.$ The argument $z$ of $\Gamma_e(z;p,q)$ arises from the holonomy of a background gauge field, under which the chiral multiplet has unit charge. The limit $p\to0$ degenerates the angle $\varphi$: while we started with $S^3$ as an $S^1\times S^1$ fibered over an interval, we are left with the angle $\chi$ fibered over the interval $\theta\in[0,\pi]$, giving us a disk.
As usual with such reductions (see e.g.~\cite{Dedushenko:2023cvd}) the emerging boundary condition on $\partial D^2$ is Neumann. The degeneration preserves a transversely holomorphic foliation (THF) \cite{gomez1980transversal} determined by the surviving complex coordinate $z_2$.
The 3d BPS partition function depends holomorphically on the modulus $q$ of $z_2$  \cite{Closset:2013vra}. Therefore
$\mathcal{M}_{p,q}\xrightarrow{p\to0}M_q\,,$
with $M_q$ the topologically $D^2\times S^1$ space with THF modulus $q.$ Reduction of the 4d $\mathcal{N}=1$ chiral multiplet gives a 3d $\mathcal{N}=2$ chiral multiplet, whose BPS partition function on $M_q$ with Neumann boundary conditions is indeed given, up to a Casimir-type factor, by  $\frac 1{(z;q)_\infty}$ \cite{Yoshida:2014ssa}.

To see how Narukawa's formula  may yield our identity in the limit $p\to 0$,
note that \cite[Prop.\ III.6]{Ruijsenaars}
\begin{equation}\label{hyprat}\lim_{\Im\omega_1\rightarrow\infty} \log\Gamma_h(x;\omega_1,\omega_2)-\left(\frac{x}{\omega_2}-\frac 12\right)\log\frac{2\pi\omega_2}{\omega_1}=\log\Gamma\left(\frac x{\omega_2}\right)-\frac 12\log(2\pi).\end{equation}
Suppose that we substitute in \eqref{narukawa} $(x,\omega_1,\omega_2)=(\ti y/2\pi,\ti\alpha/2\pi,\ti\beta/2\pi)$ and then let $\Re\alpha\rightarrow\infty$.  
Then, the $\log \Gamma_e$-term tends to $-\log(e^{-y};e^{-\beta})_\infty$ and, up to correction terms, the $\log\Gamma_h$-terms to $\log\Gamma((y-2\pi\ti n)/\beta)$.
Hence, one may expect to obtain \eqref{main-identity-convergent}. However,
 the  terms on the left of \eqref{narukawa}, as well as the $Q$-term on the right, blow up in the limit. We will remove this divergence by adding  suitable correction terms.

To motivate  our choice of regularization  we  
appeal to \eqref{shintani}.
If $\Im(\omega_1/\omega_2)\rightarrow \infty$, 
the $q$-Pochhammer symbols with $q=e^{2\pi\ti\omega_1/\omega_2}$ decay
but those with $q=e^{-2\pi\ti\omega_2/\omega_1}$
blow up.  We  approximate the divergent 
terms using (cf.\ \eqref{asymptotics-fixed-y})
$$\log (e^{-y};e^{-\beta})_\infty=-\frac{\Li_2(e^{-y})}{\beta}-\frac{\Li_1(e^{-y})}2-\frac{\Li_0(e^{-y})}{12}\,\beta+\mathcal O(\beta^2). $$
That is, for $n>0$  we modify \eqref{narukawa} by adding the correction terms
$$-\frac{\omega_1}{2\pi\ti \omega_2}\Li_2(e^{-2\pi\ti(x+n)/\omega_1})-\frac 12\Li_1(e^{-2\pi\ti (x+n)/\omega_1})-\frac {\pi\ti\omega_2}{6\omega_1}\Li_0(e^{-2\pi\ti (x+n)/\omega_1}). $$
In the case $n\leq 0$, we also use
$$\log(e^{-y-\beta};e^{-\beta})_\infty=\Li_1(e^{-y})+\log(e^{-y};e^{-\beta})_\infty,$$
to see that the appropriate correction terms are
$$ \frac{\omega_1}{2\pi\ti \omega_2}\Li_2(e^{2\pi\ti(x+n)/\omega_1})-\frac 12\Li_1(e^{2\pi\ti (x+n)/\omega_1})+
\frac {\pi\ti\omega_2}{6\omega_1}\Li_0(e^{2\pi\ti(x+n)/\omega_1}).$$
This leads us to consider the regularized series 
\begin{multline}\label{regular}S=\sum_{n\in\mathbb Z} 
\log\Gamma_h(x+n;\omega_1,\omega_2)\pm \pi\ti P(x+n;\omega_1,\omega_2)\\
\pm \frac{\omega_1}{2\pi\ti \omega_2}\Li_2(e^{\pm 2\pi\ti(x+n)/\omega_1})-\frac 12\Li_1(e^{\pm 2\pi\ti (x+n)/\omega_1})\pm \frac {\pi\ti\omega_2}{6\omega_1}\Li_0(e^{\pm 2\pi\ti(x+n)/\omega_1}),
 \end{multline}
where we choose  plus signs for $n\leq 0$ and minus signs for $n>0$. 
We will show that the term-wise limit $\Im(\omega_1)\rightarrow\infty$ of \eqref{regular}
 is \eqref{main-identity-convergent}. 

Using \eqref{li-exp} it is straightforward to check that, regardless of the sign,
\begin{multline*}  \pm \pi\ti P(x;\omega_1,\omega_2)
\pm \frac{\omega_1}{2\pi\ti \omega_2}\Li_2(e^{\pm 2\pi\ti x/\omega_1})-\frac 12\Li_1(e^{\pm 2\pi\ti x/\omega_1})\pm \frac {\pi\ti\omega_2}{6\omega_1}\Li_0(e^{\pm 2\pi\ti x/\omega_1})\\
=\frac{x}{\omega_2}-\frac{\omega_2}{12x}+\left(\frac 12-\frac{x}{\omega_2}\right)\log\left(\frac{2\pi x}{\omega_1}\right)+\mathcal O(\omega_1^{-1}).
\end{multline*}
Using also \eqref{hyprat}, it follows that the terms in \eqref{regular} tend to
$f((x+n)/\omega_2)$, where  $f$ is as in \eqref{stirling}.

On the other hand, it follows from \eqref{narukawa} that 
$$S=\log \Gamma_e(e^{2\pi\ti x};e^{2\pi\ti\omega_1},e^{2\pi\ti\omega_2})
+R, $$
where 
\begin{align}
\nonumber R&=-2\pi\ti Q(x;\omega_1,\omega_2)+\frac{\omega_1}{2\pi\ti\omega_2}\sum_{k=0}^\infty
\left(-\Li_2(e^{-2\pi\ti (x+k+1)/\omega_1})+\Li_2(e^{2\pi\ti (x-k)/\omega_1})\right)\\
\nonumber&\quad-\frac 12\sum_{k=0}^\infty
\left(\Li_1(e^{-2\pi\ti (x+k+1)/\omega_1})+\Li_1(e^{2\pi\ti (x-k)/\omega_1})\right)\\
\label{correction}&\quad+\frac{\pi\ti\omega_2}{6\omega_1}\sum_{k=0}^\infty
\left(-\Li_0(e^{-2\pi\ti (x+k+1)/\omega_1})+\Li_0(e^{2\pi\ti (x-k)/\omega_1})\right).
\end{align}
To compute this correction term we use the modular transformation  for Jacobi theta functions, in the form
\begin{multline}\label{theta-modular}(e^{-(x+1)\beta};e^{-\beta})_\infty(e^{x\beta};e^{-\beta})_\infty\\
= e^{-\frac{\pi^2}{3\beta}-\pi\ti\left(x+\frac 12\right)+\left(\frac{x^2}2+\frac x2+\frac1{12}\right)\beta}
(e^{2\pi \ti x};e^{-4\pi^2/\beta})_\infty(e^{-2\pi\ti x-4\pi^2/\beta};e^{-4\pi^2/\beta})_\infty\end{multline}
or,  taking logarithms,
\begin{multline}\label{logsum}\sum_{k=0}^\infty\left(\Li_1(e^{-(x+k+1)\beta})+ \Li_1(e^{\beta(x-k)})\right)
= \frac{\pi^2}{3\beta}+\pi\ti\left( x+\frac12\right)-\left(\frac{x^2}2+\frac x2+\frac 1{12}\right)\beta\\
+\sum_{k=0}^\infty\left(\Li_1(e^{2\pi\ti x-4\pi^2k/\beta})+\Li_1(e^{-2\pi\ti x-4\pi^2(k+1)/\beta})\right)
.\end{multline}
In particular, as $\beta\rightarrow 0$,
\begin{subequations}\label{polylogsums} 
\begin{equation}\sum_{k=0}^\infty\left(\Li_1(e^{-(x+k+1)\beta})+ \Li_1(e^{\beta(x-k)})\right)=
\frac{\pi^2}{3\beta}+\pi\ti\left( x+\frac12\right)+\Li_1(e^{2\pi\ti x})+\mathcal O(\beta).
 \end{equation}
Differentiating \eqref{logsum} in $x$ similarly leads to
\begin{equation}\beta\sum_{k=0}^\infty\left(-\Li_0(e^{-(x+k+1)\beta})+ \Li_0(e^{\beta(x-k)})\right)\\
= \pi\ti+2\pi\ti \Li_0(e^{2\pi\ti x})+\mathcal O(\beta)\end{equation}
and, taking the primitive function that vanishes for $x=-1/2$,
\begin{multline}\frac 1\beta \sum_{k=0}^\infty\left(-\Li_2(e^{-(x+k+1)\beta})+ \Li_2(e^{\beta(x-k)})\right)\\
=
\frac{\pi^2(x+1/2)}{3\beta}
+\pi\ti\left(\frac{x^2}2+\frac x2+\frac 1{12}\right)+\frac {\Li_2(e^{2\pi\ti x})}{2\pi\ti}+\mathcal O(\beta).\end{multline}
\end{subequations}
Inserting the expressions \eqref{polylogsums}
into \eqref{correction}, a straight-forward computation 
shows that  only the polylogarithmic terms survive in the limit. That is,
$$\lim_{\Im\omega_1\rightarrow\infty}R= 
\frac{\Li_2(e^{2\pi\ti x})}{2\pi\ti\omega_2}-\frac{\Li_1(e^{2\pi\ti x})}{2}
+\frac{2\pi\ti \omega_2\Li_0(e^{2\pi\ti x})}{12}.
$$
We conclude that the termwise limit $\Im \omega_1\rightarrow\infty$ of \eqref{regular} is
$$\sum_{n\in\mathbb Z}f\left(\frac{x+n}{\omega_2}\right)=-\log(e^{2\pi\ti x};e^{2\pi\ti\omega_2})_\infty+ \frac{\Li_2(e^{2\pi\ti x})}{2\pi\ti\omega_2}-\frac{\Li_1(e^{2\pi\ti x})}{2}
+\frac{2\pi\ti \omega_2\Li_0(e^{2\pi\ti x})}{12},$$
which is equivalent to \eqref{main-identity-convergent}. To rigorously prove \eqref{main-identity-convergent} from \eqref{narukawa}, one needs to justify interchanging the limit and sum. We will not work out the details since we have already given two proofs of \eqref{main-identity-convergent}.

\printbibliography

\end{document}